\documentclass[12pt,amstex]{amsart}

\usepackage{mathptmx}
\usepackage{mathrsfs}

\usepackage{verbatim}
\usepackage{url}
\usepackage[all]{xy}
\usepackage{color}

\usepackage[colorlinks=true,citecolor=blue]{hyperref}

\usepackage{stmaryrd}
\usepackage{epsfig}
\usepackage{amsmath}
\usepackage{amssymb}

\usepackage{amscd}
\usepackage{graphicx}
\usepackage{pstricks}

\theoremstyle{plain}
\newtheorem{thm}{Theorem}[section]
\newtheorem{lem}[thm]{Lemma}

\newtheorem{cor}[thm]{Corollary}
\theoremstyle{definition}
\newtheorem{de}[thm]{Definition}

\theoremstyle{remark}
\newtheorem{rem}[thm]{Remark}

\numberwithin{equation}{section}

\def \N {\mathbb N}

\def \Z {\mathbb Z}

\def \F {\mathcal{F}}

\def \X {\mathcal{X}}

\def \O {\mathcal{O}}

\def \n {\mathcal{N}}

\def \PG {{{\bf P}\Gamma}}

\def \D {\Delta}

\begin{document}

\title[]{Topological correspondence of multiple ergodic averages of nilpotent group actions}

\author{Wen Huang}
\author{Song Shao}
\author{Xiangdong Ye}

\address{Wu Wen-Tsun Key Laboratory of Mathematics, USTC, Chinese Academy of Sciences and
Department of Mathematics, University of Science and Technology of China,
Hefei, Anhui, 230026, P.R. China.}

\email{wenh@mail.ustc.edu.cn}\email{songshao@ustc.edu.cn}\email{yexd@ustc.edu.cn}



\thanks{Huang is partially supported by NNSF for Distinguished Young Schooler (11225105),
and all authors are supported by NNSF of China (11371339, 11431012, 11571335).}

\date{Feb. 1, 2016}
\date{May, 4, 2016}
\date{June 10, 2016}
\date{Oct. 7, 2016}

\begin{abstract}
Let $(X,\Gamma)$ be a topological system, where $\Gamma$ is a nilpotent
group generated by $T_1,\ldots, T_d$  such that for each $T\in \Gamma$, $T\neq  e_\Gamma$,
$(X,T)$ is weakly mixing and minimal. For $d,k\in \N$,
let $p_{i,j}(n), 1\le i\le k, 1\le j\le d$
be polynomials with rational coefficients taking integer values on the integers and $p_{i,j}(0)=0$. We show that if the expressions
$g_i(n)=T_1^{p_{i,1}(n)}\cdots T_d^{p_{i,d}(n)}$
depends nontrivially on $n$ for $i=1,2,\cdots,k$, and for all $i\neq j\in \{1,2,\ldots,k\}$ the expressions $g_i(n)g_j(n)^{-1}$ depend nontrivially on $n$,
then there is a residual set $X_0$ of $X$ such that for all $x\in X_0$
\begin{equation*}
  \{(g_1(n)x, g_2(n)x,\ldots, g_k(n)x)\in X^k:n\in \mathbb{Z}\}
\end{equation*}
is dense in $X^k$.

\end{abstract}

\maketitle

\markboth{Topological correspondence of multiple ergodic averages}{W. Huang, S. Shao and X. Ye}




\section{Introduction}

Measurable dynamics and topological dynamics are two sister branches of the theory of dynamical systems,
who use similar words to describe different but parallel notions in their respective theories. The surprising fact
is that many of the corresponding results are rather similar though the proofs may be quite different. For the interplay
between measurable and topological dynamics, we refer to the survey by Glasner and Weiss \cite{GW06}. In this paper,
we study the topological analogue of multiple ergodic averages of weakly mixing systems under nilpotent group actions.

\subsection{Main results}\
\medskip

Motivated by the work of Furstenberg on the multiple recurrence theorem \cite{F77}, in his pioneer work Glasner
presented in \cite{G94} the counterpart of \cite{F77} in topological dynamics. As it is said
in \cite{G94}: ``The basic problem in both the measure theoretical and the topological theory is roughly
the following: given a system $(X,T)$ (ergodic or minimal) and a positive integer $n$, describe the most
general relation that holds for $(n+1)$-tuples $(x,Tx,T^2x,\ldots, T^nx)$ in the product space
$X\times X\times \ldots \times X$ ($n+1$ times).'' One of the main results in \cite{G94} is that: for a topologically weakly mixing
and minimal system $(X,T)$, there is a dense $G_\delta$ subset $X_0$ such that for each
$x\in X_0$, $(T^nx, \ldots, T^{dn}x)$ is dense in $X^d$. Note that a different proof of Glasner's theorem
on weakly mixing systems was presented in \cite{KO, Moo}.

\medskip

In this paper we extend this result to a much broader setting. Let $\mathcal{P}$ be the collection of all polynomials
with rational coefficients taking
integer values on the integers, $\mathcal{P}_0$ be the collection of elements $p$ of $\mathcal{P}$ with $p(0)=0$, and
$\mathcal{P}_0^*$ be  the collection of non-constant elements of $\mathcal{P}_0$.
The main results of this paper are the following:

\begin{thm}\label{main-thm}
Let $(X,\Gamma)$ be a topological system, where $\Gamma$ is a nilpotent group  such that for each $T\in \Gamma$,
$T\neq e_\Gamma$, is weakly mixing and minimal. For $d,k\in \N$ let $T_1,\ldots,T_d\in \Gamma$,
$\{p_{i,j}(n)\}_{1\le i\le k, 1\le j\le d}\in \mathcal{P}_0$ 
such that the expression
\begin{equation*}
  g_i(n)=T_1^{p_{i,1}(n)}\cdots T_d^{p_{i,d}(n)}
\end{equation*}
depends nontrivially on $n$ for $i=1,2,\ldots, k$, and for all $i\neq j\in \{1,2,\ldots,k\}$ the expressions $g_i(n)g_j(n)^{-1}$
depend nontrivially on $n$.
Then there is a dense $G_\delta$ subset $X_0$ of $X$ such that for all $x\in X_0$
\begin{equation*}
\{(g_1(n)x,\ldots,g_k(n)x):n\in\mathbb{Z}\}
\end{equation*}
is dense in $X^k$.
\end{thm}

We remark that the non-degeneracy conditions stated in the above theorem is also necessary.
Note that we say that $g(n)$ {\it depends nontrivially on $n$}, if $g(n)$ is a nonconstant mapping from $\Z$ into $\Gamma$,
and $g_1(n), g_2(n)$ are {\it distinct} if $g_1(n)g_2^{-1}(n)$ depends nontrivially on $n$.
When $\Gamma$ is abelian, one has that
\begin{equation*}
  g_i(n)g_j(n)^{-1}=T_1^{p_{i,1}(n)-p_{j,1}(n)}\cdots T_d^{p_{i,d}(n)-p_{j,d}(n)}.
\end{equation*}
When $\Gamma$ is nilpotent, the expressions of $g_i(n)$ and $g_i(n)g_j(n)^{-1}$ depend on the Malcev basis of $\Gamma$
(see Section \ref{section-nil}).

Taking $\Gamma=\Z$ and $d=1$ in Theorem \ref{main-thm}, we have the result for one transformation.

\begin{thm}\label{main-1}
Let $(X,T)$ be a weakly mixing minimal system and $p_1, \ldots, p_d\in \mathcal{P}_0^*$ be
distinct polynomials. 
Then there is a dense $G_\delta$ subset $X_0$ of $X$ such that for any $x\in X_0$
$$\{(T^{p_1(n)}(x), \ldots, T^{p_d(n)}(x)):n\in\mathbb{Z}\}$$
is dense in $X^d$.
\end{thm}

\subsection{Multiple ergodic averages for weakly mixing systems}\
\medskip

Now we state some corresponding results in ergodic theory.
For a weakly mixing system, Bergelson and Leibman \cite[Theorem D]{BL96} showed the following result:
Let $(X,\X,\mu,\Gamma)$ be a measure preserving system, where $\Gamma$ is an abelian group such
that for each $T\in \Gamma$, $T\neq e_\Gamma$, is weakly mixing. For $d,k\in \N$, let $T_1,\ldots,T_d\in \Gamma$,
and $p_{i,j}\in \mathcal{P}_0, 1\le i\le k, 1\le j\le d$  
such that the expressions $g_i(n)$ satisfies the non-degeneracy conditions stated in Theorem \ref{main-thm}.
Then for any $f_1,\ldots, f_k\in L^\infty(X,\mu)$,
\begin{equation*}
  \lim_{N\to\infty}\Big\|\frac{1}{N}\sum_{n=0}^{N-1}\prod_{i=1}^k f_i(T_1^{p_{i,1}(n)}\cdots T_d^{p_{i,d}(n)}x)-\prod_{i=1}^k\int_Xf_i(x)d\mu\Big\|_{L^2}=0.
\end{equation*}
Related results were proved for nilpotent group actions by Leibman \cite[Theorem 11.15]{Leibman98}.
Note that topological and measurable multiple recurrent theorems under nilpotent group actions were
also studied in \cite{BL96, Leibman94, ZK}.

\medskip

It is natural to conjecture that the result above is still valid for the pointwise convergence, i.e. for weakly mixing
nilpotent group actions, we conjecture that for a subset $X_0$ with full measure, and each $x\in X_0$
the averages $$\frac{1}{N}\sum_{n=1}^N\prod_{j=1}^k f_j(T_1^{p_{1,j}(n)}T_2^{p_{2,j}(n)}\cdots T_d^{p_{d,j}(n)}x)$$
converge to the product of the integrals if $g_i(n)=T_1^{p_{i,1}(n)}\cdots T_d^{p_{i,d}(n)}, \ i=1,2,\ldots, k$
satisfy the obvious non-degeneracy condition. In this paper in some sense we add an evidence to support this conjecture, i.e. we give a
topological correspondence of multiple ergodic averages of nilpotent weakly mixing group actions.

\medskip

Finally we say a few more words on multiple ergodic averages.
Followed from Furstenberg's beautiful work \cite{F77} on the dynamical proof of Szemer\'{e}di's theorem in 1977,
problems concerning the convergence of
multiple ergodic averages (or called ``non- conventional averages'' \cite{F88, F10}) in $L^2$ or pointwisely attract a lot of attention.
Nowadays we have rich results for the $L^2$-norm convergence \cite{HK05, Tao, Walsh, Z}.
On the other hand, there are a few results related to the pointwise convergence of multiple ergodic averages.
Bourgain showed that the limit of $\frac{1}{N}\sum_{n=0}^{N-1} f(T^{p(n)}x)$ exists a.e. for all integer
valued polynomials $p(n)$ and $f\in L^{p}(X,\X,\mu)$ with $p>1$ \cite{B89}, and the averages
$\frac{1}{N} \sum_{n=0}^{N-1}f_1(T^{a_1n}x)f_2(T^{a_2n}x)$ converge a.e. for $a_1, a_2\in \Z$
and all $f_1,f_2$ in $L^\infty(X,\X,\mu)$. Huang, Shao and Ye \cite{HSY} showed that for all distal systems,
$ \frac 1 N\sum_{n=0}^{N-1}f_1(T^nx)\ldots f_d(T^{dn}x)$ exists a.e. for all $f_1$, $\ldots$,
$f_d$ $\in$ $L^\infty(X,\X,\mu)$, where $d\in \N$. Very recently,
 Donoso and Sun in \cite{DS16} generalized the above result to commuting distal transformations.

\medskip

In \cite{G94} it was also showed that, up to a canonically defined proximal
extension, a characteristic family for $T\times T^2\times \ldots \times T^n$, is the family of canonical
PI flows of class $n-1$. In particular, when $(X, T)$ is minimal and
distal, most $T\times T^2\times \ldots \times T^n$ orbit closures of points $(x,x,\ldots,x)$ in the diagonal
of $X^n$ are lifts of the corresponding orbit closures in the largest class-$(n-1)$ factor.
In view of this fact and the recent progress related to the convergence of multiple ergodic averages,
it is an interesting question how to formulate and prove the counterpart in topological dynamics for nilpotent group actions.
In this paper we have investigated the weak mixing case, and we plan to treat
the non-weakly mixing system in the future research.

\subsection{Strategy of the proofs and further results}\
\medskip

To prove Theorem \ref{main-thm} we use PET-induction, which was introduced by Bergelson in \cite{Bergelson87}.
The PET-induction we use in the current paper is due to Leibman \cite{Leibman94}. The basic idea of
this induction is that: we associate any finite collection of polynomials a "complexity", and reduce the complexity
at some step to the trivial one. Note that in some step, the cardinal number of the collection may increase while the complexity decreases.

It is easy to show  that to prove Theorem \ref{main-thm}, it is equivalent to
prove that for any given non-empty open subsets $U, V_1,\ldots, V_k$ of $X$,
\begin{equation}\label{syndetic-1}
\{n\in \Z: U\cap (g_1(n)^{-1}V_1\cap \ldots \cap g_k(n)^{-1}V_k)\not=\emptyset\}
\end{equation}
is infinite (Lemma \ref{diagonal1}). Basically, this can be done by proving a proposition related to the weakly mixing property
(Lemma \ref{lem-KO}) and the fact that for all non-empty open sets $U_1,\ldots, U_k$ and $V_1,\ldots,V_k$ of $X$
\begin{equation}\label{st-1}
  \{n\in \mathbb{Z}: U_1\times \ldots \times U_k \cap g_1(n)^{-1}\times \ldots\times g_k(n)^{-1}(V_1\times \ldots \times V_k)\neq \emptyset\}
\end{equation}
is infinite. Practically, when doing this, we find that if in the collection of polynomials there are linear elements and other non-linear
elements, the argument will be very much involved. To overcome this difficulty, we actually show that for non-empty open subsets $U,V$
and a $\Gamma$-polynomial $g(n)$, $\{n\in \mathbb{Z}: U \cap g(n)^{-1}(V)\neq \emptyset\}$ is thickly-syndetic. Since the family of thickly-syndetic
subsets is a filter, this implies (\ref{st-1}). To prove this, we
need to show that (\ref{syndetic-1}) is syndetic. This means that the proof of Theorem \ref{main-thm} is achieved by showing the following stronger result:

\begin{thm}\label{stronger-thm}
Let $(X,\Gamma)$ be a topological system, where $\Gamma$ is a nilpotent group  such that for each $T\in \Gamma$,
$T\neq e_\Gamma$, is weakly mixing and minimal. For $d,k\in \N$ let $T_1,\ldots,T_d\in \Gamma$,
$\{p_{i,j}(n)\}_{1\le i\le k, 1\le j\le d}\in \mathcal{P}_0$ 
such that the expression
\begin{equation*}
  g_i(n)=T_1^{p_{i,1}(n)}\cdots T_d^{p_{i,d}(n)}
\end{equation*}
depends nontrivially on $n$ for $i=1,2,\ldots, k$, and for all $i\neq j\in \{1,2,\ldots,k\}$ the expressions $g_i(n)g_j(n)^{-1}$
depend nontrivially on $n$. Then for all non-empty open sets $U_1,\ldots, U_k$ and $V_1,\ldots,V_k$ of $X$
\begin{equation*}
  \{n\in \mathbb{Z}: U_1\times \ldots \times U_k \cap g_1(n)^{-1}\times \ldots\times g_k(n)^{-1}(V_1\times \ldots \times V_k)\neq \emptyset\}
\end{equation*}
is a thickly-syndetic set, and
\begin{equation*}
\{n\in \mathbb{Z}:  U\cap (g_1(n)^{-1}V_1\cap \ldots \cap g_k(n)^{-1}V_k)\not=\emptyset\}
\end{equation*}
is a syndetic set.
\end{thm}

We note that when doing the induction procedure, we need to check the non-degeneracy conditions of the reduced collection.
We find that the known results are not enough to guarantee them, and we should prove additional lemmas whose proofs are presented
in Subsection \ref{add-1}.

\medskip

After we introduce PET-induction in Section \ref{section-nil}, we will explain the main ideas of the proof via proving Theorem \ref{main-1}.
As an application of Theorem \ref{main-thm} we have

\begin{thm}\label{ye00}
Let $(X,\Gamma)$ be a topological system, where $\Gamma$ is a nilpotent group  such that for each $T\in \Gamma$,
$T\neq e_\Gamma$, is weakly mixing and minimal. For $k\in \N$ let $T_1,\ldots,T_k\in \Gamma$,
$\{p_{i}(n)\}_{1\le i\le k}\in \mathcal{P}_0$ 
such that the expression $g(n)=T_1^{p_{1}(n)}\cdots T_k^{p_{k}(n)}$ depends nontrivially on $n$.
Then there is a dense
$G_\delta$ subset $X_0$ of $X$ such that for each $x\in X_0$ and each non-empty open subset $U$ of $X$
$$N_g(x,U):=\{n\in\Z: g(n)x\in U\}$$
is piecewise syndetic.
\end{thm}

\begin{rem} It is easy to see that to show the above theorems, we may assume that the coefficients of the polynomials involved
are integers.
\end{rem}

\subsection{Organization of the paper}

We organize the paper as follows. In Section \ref{section-preliminary} we introduce some basic notions
and facts we need in the paper. In Section \ref{section-nil}, we recall the PET-induction
for nilpotent group actions. In Section \ref{section-example},
we show some examples and outline the proof of Theorem \ref{main-1}, which  provides the main
ideas how to prove Theorem \ref{stronger-thm}. In the final section, we give the complete
proof of Theorem \ref{stronger-thm} and Theorem \ref{ye00}.

\medskip
\noindent{\bf Acknowledgments:}
We would like to thank the referee for very useful suggestion which  makes the paper more readable.

\section{Preliminary}\label{section-preliminary}

\subsection{Topological transformation groups}\
\medskip

A {\em topological dynamical system} (t.d.s. for short) is a triple
$\X=(X, \Gamma, \Pi)$, where $X$ is a compact metric space, $\Gamma$ is a
Hausdorff topological group with the unit $e_\Gamma$ and $\Pi: \Gamma\times X\rightarrow X$ is a
continuous map such that $\Pi(e_\Gamma,x)=x$ and
$\Pi(s,\Pi(t,x))=\Pi(st,x)$. We shall fix $\Gamma$ and suppress the
action symbol. In many references, $(X,\Gamma)$ is also called a {\em
topological transformation group} or a {\em flow}.

\medskip

Let $(X,\Gamma)$ be a t.d.s. and $x\in X$, then $\O(x,\Gamma)$ denotes the
{\em orbit} of $x$, which is also denoted by $\Gamma x$. A subset
$A\subseteq X$ is called {\em invariant} if $t a\subseteq A$ for all
$a\in A$ and $t\in \Gamma$. When $Y\subseteq X$ is a closed and
$\Gamma$-invariant subset of the system $(X, \Gamma)$ we say that the system
$(Y, \Gamma)$ is a {\em subsystem} of $(X, \Gamma)$. If $(X, \Gamma)$ and $(Y,
\Gamma)$ are two dynamical systems their {\em product system} is the
system $(X \times Y, \Gamma)$, where $t(x, y) = (tx, ty)$.

\medskip

A system $(X,\Gamma)$ is called {\em minimal} if $X$ contains no proper
non-empty closed invariant subsets. $(X,\Gamma)$ is called {\em transitive} if
every non-empty invariant open subset of $X$ is dense. An example of a
transitive system is the {\em point-transitive} system, which is a
system with a dense orbit. It is easy to verify that a system is
minimal iff every orbit is dense. A point $x\in X$ is called a {\em minimal} point if $(\overline{\O(x,\Gamma)}, \Gamma)$ is a minimal subsystem.
A system $(X,\Gamma)$ is {\em weakly mixing} if the product system $(X
\times X,\Gamma)$ is transitive.

\subsection{Some important subsets of integers}\
\medskip

A subset $S$ of $\Z$  is {\it syndetic} if it has bounded gaps,
i.e. there is $N\in \N$ such that $\{i,i+1,\cdots,i+N\} \cap S \neq
\emptyset$ for every $i \in {\Z}$. $S$ is {\it thick} if it
contains arbitrarily long runs of integers, i.e. there is a
subsequence $\{n_i\}_{i=1}^\infty$ of $\Z$ with $|n_{i+1}|>|n_i|$ for any $i\in\N$ such that
$S\supset \bigcup_{i=1}^\infty \{n_i, n_i+1, \ldots, n_i+i\}$.
Some dynamical properties can be interrupted by using the notions of
syndetic or thick subsets. For example, a classic result of
Gottschalk and Hedlund stated that $x$ is a minimal point if and only if
$$N(x,U)=\{n\in\Z: T^nx\in U\}$$ is syndetic for any neighborhood $U$ of $x$,
and by Furstenberg \cite{F67} a topological system
$(X,T)$ is weakly mixing if and only if
$$N(U,V)=\{n\in\Z: U\cap T^{-n}V\neq \emptyset \}$$ is thick for any
non-empty open subsets $U,V$ of $X$.

\medskip

A subset $S$ is called {\em thickly-syndetic} if for every $N\in\N$ the positions where length $N$ runs begin form a syndetic set.
A subset $S$ of $\Z$ is {\it piecewise syndetic} if it is an
intersection of a syndetic set with a thick set.

Note that the set of all thickly-syndetic sets is a filter, i.e. the intersection of two thickly-syndetic
sets is still a thickly-syndetic set (see \cite{F} for more details).

\medskip

The following lemma will be used in the sequel.

\begin{lem}\cite[Theorem 4.7.]{HY02}\label{lem-hy}
For a minimal and weakly mixing system
$(X,T)$,
$$N(U,V)=\{n\in\Z: U\cap T^{-n}V\neq \emptyset \}$$ is thickly-syndetic for any
nonempty open subsets $U,V$ of $X$.
\end{lem}

Since the collection of all thickly syndetic sets is a filter, one consequence of Lemma~ \ref{lem-hy} is:

\begin{cor}\label{cor-hy}
Let $d\in \N$ and $(X,T_1),\ldots, (X,T_d)$ be weakly mixing and minimal systems. Then
$(X_1\times \ldots\times X_d, T_1\times \ldots \times T_d)$ is weakly mixing.
\end{cor}

\subsection{Some notions and useful lemmas}\
\medskip

\subsubsection{Notations}\label{notation}\
\medskip

Let $(X,\Gamma)$ be a t.d.s., $\Gamma$ be a group, and $d,k\in \N$. Let $T_1,\ldots,T_d\in \Gamma$,
$p_{i,j}(n)\in \mathcal{P}_0, 1\le i\le k, 1\le j\le d$, and let
\begin{equation*}
  g_i(n)=T_1^{p_{i,1}(n)}\cdots T_d^{p_{i,d}(n)}, \quad i=1,2,\ldots, k.
\end{equation*}

We will fix the above notation in the rest of this section.

\subsubsection{$\{g_1,\ldots,g_k\}_\D$-transitivity}

\begin{de}
We say $(X,\Gamma)$ is {\em $\{g_1,\ldots,g_k\}_\D$-transitive} if for any given non-empty open subsets $U, V_1,\ldots, V_k$ of $X$,
\begin{equation*}
  \{n\in \Z: U\cap (g_1(n)^{-1}V_1\cap \ldots \cap g_k(n)^{-1}V_k)\not=\emptyset\}
\end{equation*}
is infinite.
\end{de}

The following lemma is a generalization of an observation in \cite{Moo}.

\begin{lem}\label{diagonal1}
Let $(X,\Gamma)$ and $g_1,\ldots, g_k$ be defined as in subsection \ref{notation}.
Then there is a dense $G_\delta$ set $X_0$ of $X$ such that for all $x\in X_0$
\begin{equation*}
  \{(g_1(n)x, g_2(n)x,\ldots,g_k(n)x)\in X^k:n\in \mathbb{Z}\}
\end{equation*}
is dense in $X^k$ if and only if it is $\{g_1,\ldots,g_k\}_\D$-transitive.
\end{lem}

\begin{proof}
One direction is obvious. And now assume that for any given non-empty open sets $U, V_1,\ldots, V_k$ of $X$, $\{n\in \Z:U\cap (g_1(n)^{-1}V_1\cap \ldots \cap g_k(n)^{-1}V_k)\not=\emptyset\}$
is infinite.

Let $\F$ be a countable base of $X$, and let
\begin{equation*}
  X_0=\bigcap_{V_1,\ldots,V_k\in \F}\bigcup_{n\in \Z} g_1(n)^{-1}V_1\cap \ldots \cap g_k(n)^{-1}V_k.
\end{equation*}
Then it is easy to see that the dense $G_\delta$ subset $X_0$ is what we need.
\end{proof}

Hence by Lemma \ref{diagonal1}, Theorem \ref{main-thm} can be restated as: Assume all the conditions in Theorem \ref{main-thm}, then $(X,\Gamma)$ is $\{g_1,\ldots,g_k\}_\D$-transitive.


\subsubsection{$\{g_1,\ldots,g_k\}$-transitivity}

\begin{de}
$(X,\Gamma)$ is {\em $\{g_1,\ldots, g_k\}$-transitive} if for all non-empty open sets $U_1,\ldots, U_d$ and $V_1,\ldots,V_d$ of $X$,
\begin{equation*}
  \{n\in \Z: U_1\times \ldots \times U_k \cap g_1^{-1}(n)\times \ldots\times g_k^{-1}(n)(V_1\times \ldots \times V_d)\neq \emptyset\}
\end{equation*}
is infinite.
\end{de}

The following lemma is a generalization of Lemma 3 of \cite{KO}.

\begin{lem}\label{lem-KO} \label{freedom}
Let $(X,\Gamma)$ and $g_1,\ldots, g_k$ be defined as in subsection \ref{notation} and $T\in \Gamma$.
If $({X}, \Gamma)$ is $\{g_1,\ldots, g_k\}$-transitive, then for any non-empty open sets $V_1,\ldots ,V_k$ of $X$ and any
subsequence $\{r(n)\}_{n=0}^\infty$ of natural numbers, there is a sequence of integers $\{k_n\}_{n=0}^\infty$ such that
$|k_0|>r(0)$, $|k_n|>|k_{n-1}|+r(|k_{n-1}|)$  for all $n\ge 1$, and for each $i\in \{1,2,\ldots,k\}$, there is a descending
sequence $\{V_i^{(n)}\}_{n=0}^\infty$ of non-empty open subsets of $V_i$ such that for each $n\ge 0$ one has that
\begin{equation*}
  g_i(k_j)T^{-j}V_i^{(n)}\subseteq V_i, \quad \text{for all}\quad 0\le j\le n.
\end{equation*}
\end{lem}


\begin{proof}
Let $V_1,\ldots ,V_k$ be non-empty open subsets of $X$. Since $(X,\Gamma)$ is $\{g_1,\ldots, g_k\}$-transitive, there is some $k_0$ with $|k_0|>r(0)$ such that
\begin{equation*}
  V_1\times \ldots \times V_k \cap g_1^{-1}(k_0)\times \ldots\times g_k^{-1}(k_0)(V_1\times \ldots \times V_k)\neq \emptyset.
\end{equation*}
That is, $g_i^{-1}(k_0)V_i\cap V_i\neq \emptyset$ for all $i=1,\ldots,k$. Put $V_i^{(0)}=g_i^{-1}(k_0)V_i\cap V_i$ for all $i=1,\ldots,k$ to complete the base step.

Now assume that for $n\ge 1$ we have found numbers $k_0,k_1,\ldots, k_{n-1}$ and for each $i=1,\ldots,k$, we have non-empty open subsets $V_i\supseteq V_i^{(0)}\supseteq V_i^{(1)}\ldots \supseteq V_i^{(n-1)}$ such that $|k_0|>r(0)$, and for each $m=1,\ldots, n-1$
one has $|k_m|>|k_{m-1}|+r(|k_{m-1}|)$ and
\begin{equation}\label{a2}
  g_i(k_j)T^{-j}V_i^{(m)}\subseteq V_i, \quad \text{for all}\quad 0\le j\le m.
\end{equation}

For $i=1,\ldots, k$, let $U_i=T^{-n}(V_i^{(n-1)})$. Since $({X},\Gamma)$ is $\{g_1,\ldots, g_k\}$-transitive, there is some $k_n\in \mathbb{Z}$ such that $|k_n|>|k_{n-1}|+r(|k_{n-1}|)$ and
\begin{equation*}
  U_1\times \ldots \times U_k \cap g_1^{-1}(k_n)\times \ldots\times g_k^{-1}(k_n)(V_1\times \ldots \times V_k)\neq \emptyset.
\end{equation*}
That is, $g_i^{-1}(k_n)V_i\cap U_i\neq \emptyset$ or $V_i\cap g_i(k_n)U_i\neq \emptyset$ for all $i=1,\ldots,k$.

Then for $i=1,\ldots, k$,
\begin{equation*}
  g_i(k_n)U_i\cap V_i=g_i(k_n)T^{-n}V_i^{(n-1)}\cap V_i\not=\emptyset.
\end{equation*}
Let
$$V_i^{(n)}=V_i^{(n-1)}\cap \big(g_i(k_n)T^{-n}\big)^{-1}V_i.$$
Then $V_i^{(n)}\subseteq V_i^{(n-1)}$ is a non-empty open set and clearly
$$g_i(k_n)T^{-n} V_i^{(n)}\subseteq V_i.$$
Since $V_i^{(n)}\subseteq V_i^{(n-1)}$, (\ref{a2}) still holds for $V_i^{(n)}$. Hence we finish our induction. The proof of the lemma is completed.
\end{proof}

\begin{rem}
By the proof of Lemma \ref{freedom}, if for all non-empty open sets $U_1,\ldots, U_d$ and $V_1,\ldots,V_d$ of $X$, the set $
  \{n\in \Z: U_1\times \ldots \times U_k \cap g_1^{-1}(n)\times \ldots\times g_k^{-1}(n)(V_1\times \ldots \times V_d)\neq \emptyset\}
$ contains infinitely many positive integers, then we may require that $\{k_n\}_{n=0}^\infty\subseteq \N$ in Lemma \ref{freedom}.
\end{rem}

\subsubsection{$\{g_1,\ldots,g_k\}_\D$-syndetic transitivity and $\{g_1,\ldots,g_k\}$-thickly-syndetic transitivity}\
\medskip

We will need the following definitions.

\begin{de}
We say $(X,\Gamma)$ is
\begin{enumerate}
\item {\em $\{g_1,\ldots,g_k\}_\D$-syndetic transitive}, if for any given non-empty open sets $U, V_1,\ldots, V_k$ of $X$,
\begin{equation*}
\{n\in \mathbb{Z}:  U\cap (g_1(n)^{-1}V_1\cap \ldots \cap g_k(n)^{-1}V_k)\not=\emptyset\}
\end{equation*}
is a syndetic set.

\item {\em $\{g_1,\ldots,g_k\}$-thickly-syndetic transitive} if for all non-empty open sets $U_1,\ldots, U_k$ and $V_1,\ldots,V_k$ of $X$,
\begin{equation*}
  \{n\in \mathbb{Z}: U_1\times \ldots \times U_k \cap g_1(n)^{-1}\times \ldots\times g_k(n)^{-1}(V_1\times \ldots \times V_k)\neq \emptyset\}
\end{equation*}
is a thickly-syndetic set.
\end{enumerate}
\end{de}

It is clear that $\{g_1,\ldots,g_k\}_\D$-syndetic transitivity implies $\{g_1,\ldots,g_k\}_\D$-transitivity, and $\{g_1,\ldots,g_k\}$-thickly-syndetic transitivity implies $\{g_1,\ldots,g_k\}$-transitivity.

\section{Nilpotent groups and PET-induction }\label{section-nil}

To prove Theorem \ref{stronger-thm}, we need some basic results on nilpotent groups and PET-induction. In this section, we cite the basic results
related to nilpotent groups from \cite{Leibman94}, which will be needed in the inductive part of the proof of Theorem \ref{stronger-thm}.


\medskip

In the sequel, let $\Gamma$ denote a finitely generated nilpotent group without torsion.

\subsection{Malcev basis}
\begin{thm}\label{thm-5.1}\cite{Leibman94}
Let $\Gamma$ be a finitely generated nilpotent group without torsion.
Then there exists a set of elements $\{S_1,\ldots, S_s \}$ of\ $\Gamma$ (the so called, "Malcev
basis") such that:
\begin{enumerate}
  \item  for any $1\le i<j\le s$, $[S_i, S_j]$ belongs to the subgroup of $\Gamma$ generated
by $S_1,\ldots , S_{i-1}$;
  \item every element $T$ of $\Gamma$ can be uniquely represented in the form
\begin{equation*}
  T= S_1^{r_1(T)}\cdots S_s^{r_s(T)}, \quad r_j(T)\in \Z,\quad j=1,\ldots, s;
\end{equation*}
the mapping $r:\Gamma\rightarrow \Z^s$, $r(T) = ( r_1 ( T ) , \ldots , r_s ( T ) )$, being polynomial
in the following sense: there exist polynomial mappings $R: \Z^{2s}\rightarrow \Z^s$,
$R': \Z^{s+1}\rightarrow \Z^s$ such that, for any $T,T'\in \Gamma$ and any $n\in \N$,
\begin{equation*}
  r(TT')=R(r(T),r(T')), \quad r(T^n)=R'(r(T),n).
\end{equation*}
\end{enumerate}
\end{thm}

From now on we will fix a Malcev basis $\{S_1,\ldots,S_s\}$ of $\Gamma$.

\subsection{The $\Gamma$-polynomial group}\
\medskip

An {\em integral polynomial} is a polynomial taking  integer values at the
integers.

The {\em group $\PG$} is the minimal subgroup of the group $\Gamma^\Z$ of the mappings
$\Z\rightarrow \Gamma$ which contains the constant mappings and is closed with respect
to raising to integral polynomials powers: if $g, h \in \PG$ and $p$ is an integral
polynomial, then $gh\in \PG$, where $gh(n) = g(n)h(n)$, and $g^p \in \PG$, where
$g^p(n) = g(n)^{ p(n)}$. The elements of $\PG$ are called {\em $\Gamma$-polynomials}. $\Gamma$ itself is
a subgroup of $\PG$ and is presented by the constant $\Gamma$-polynomials.

$\Gamma$-polynomials taking  the value $e_\Gamma$ at zero form a subgroup of $\PG$;
we denote it by $\PG_0$: $\PG_0 = \{g\in PG : g(0) = e_\Gamma\}$. Let $\PG_0^*=\{g\in \PG_0: g\not \equiv e_\Gamma\}$.

Every $\Gamma$-polynomial $g$ can be uniquely represented in the form
\begin{equation}\label{eq-111}
  g(n)=\prod_{j=1}^sS_j^{p_j(n)}=S_1^{p_1(n)}S_2^{p_2(n)}\cdots S_s^{p_s(n)},
\end{equation}
where $p_1,\ldots,p_s$ are integral polynomials. If $g\in \PG_0$ then $p_1,\ldots,p_s\in \mathcal{P}_0$ by Theorem~ \ref{thm-5.1}(2).


\subsection{The weight of $\Gamma$-polynomials}\
\medskip

The {\em weight}, $w(g)$, of a $\Gamma$-polynomial $g(n)=\prod_{j=1}^sS_j^{p_j(n)}$ is the
pair $(l,k)$, $l\in \{0,1,\ldots,s\}$, $k\in \Z_+$ for which $p_j=0$ for any $j > l$ and, if
$l\neq 0$, then $p_l\neq 0$ and ${\rm deg}(p_l) = k$. A weight $(l, k)$ is greater than a weight
$(l',k')$, denoted by $(l,k)>(l',k')$, if $l>l'$ or $l = l'$, $k >k'$.

For example, $S_1^n, S_1^{n^2}S_2^{n^3}, S_1^{n^6}S_2^{n^6}$ have weights $(1,1), (2,3), (2,6)$ respectively, and $(2,6)> (2,3)>(1,1)$.

Let us now define an equivalence relation on $\PG$: $g(n)=\prod_{j=1}^sS_j^{p_j(n)}$ is equivalent to $h(n)=\prod_{j=1}^sS_j^{q_j(n)}$, if $w(g)=w(h)$ and, if it is $(l,k)$,
the leading coefficients of the polynomials $p_l$ and $q_l$ coincide; we write then
$g\sim h$.
For example,
\begin{equation*}
  S_1^nS_3^{n^2}\sim S_3^{n^2+9n} \sim S_1^{n^{12}}S_2^{3n}S_3^{n^2+n}.
\end{equation*}
The {\em weight} of an equivalence class is the weight of any of its elements.

\subsection{System and its weight vector}\
\medskip

A {\em system} $A$ is a finite subset of $\PG$. For a system $A$, if we write $A=\{f_i\}_{i=1}^v$
then we require that $f_i\neq f_j$ for $1\le i\neq j\le v$. For every system $A$ we define its {\em weight vector}
$\phi(A)$ as follows. Let $w_1<w_2<\ldots <w_q$ be the set of the distinct weights of all equivalence classes
appeared in $A$. For $i=1,2,\ldots, q$, let $\phi(w_i)$ be the number of the equivalence classes of elements of $A$ with the weight $w_i$.
Let the weight vector $\phi(A)$ be
\begin{equation*}
  \phi(A)=(\phi(w_1)w_1,\phi(w_2)w_2,\ldots, \phi(w_q)w_q).
\end{equation*}

For example, let $A=\{S_1^n$, $S_1^{2n}$, $S_1^{n^2}$,
$S_1^{n}S_2^{2n^2}$, $S_1^{n^3+n^2}S_2^{2n^2+n}$,
$S_1^{n^5}S_2^{2n^2+2n}$, $S_1^{n^3}S_2^{2n^2+8n}$, \\ $S_1^{n^9+n^5+n}S_2^{n^6+n^2}$, $S_2^{2n^6+n^2}$, $S_1^{n}S_2^{3n^6+n^2}\}$. Then
$\phi(A)=\big(2(1,1),1(1,2),1(2,2), 3(2,6)\big)$.

\medskip

Let $A, A'$ be two systems. We say that $A'$ {\it precedes} a system $A$ if there exists a weight $w$ such
that $\phi(A)(w)>\phi(A')(w)$ and $\phi(A)(u)=\phi(A')(u)$ for all weight $u>w$. We denote it by
$\phi(A)\succ\phi(A')$ or $\phi(A')\prec \phi(A)$.

For example, let $w_1<w_2<\ldots <w_q$ be a sequence of weights, then
\begin{equation*}
  (a_1w_1, \ldots,a_qw_q)\succ (b_1w_1, \ldots, b_qw_q)
\end{equation*}
if and only if $(a_1,\ldots,a_q)>(b_1,\ldots,b_q)$.

\subsection{PET-induction}\
\medskip

In order to prove that a result holds for all systems $A$, we start with the system whose weight
vector is $\{1(1,1)\}$. That is, $A=\{S_1^{c_1n}\}$,
where $c_1\in \Z\setminus \{0\}$. Then let $A\subseteq \PG$ be a system whose weight vector is
greater than $\{1(1,1)\}$, and assume that for all systems $A'$ preceding $A$, we have that the
result holds for $A'$. Once we show that the result still holds for $A$, we complete the whole proof.
This procedure is called the {\em PET-induction}.

\section{Outline of the proof of Theorem \ref{main-1}}\label{section-example}

\medskip

To show the general ideas of the proof of Theorem \ref{stronger-thm}, in this section we outline the general idea how to prove Theorem \ref{main-1}. Since we deal with only one transformation in Theorem \ref{main-1}, it is relatively easy to present the basic ideas of the proof and see how PET-induction works.
In Section \ref{section-proof of Main}, we will give the complete proof of Theorem \ref{stronger-thm}.


\subsection{}
Throughout this section, $p_1, \ldots, p_d\in \mathcal{P}_0^*$ are
distinct polynomials, and $(X,T)$ is a weakly mixing minimal system.
By Lemma \ref{diagonal1}, it suffices to show that $(X,T)$ is $\{(T^{p_1(n)}, \ldots, T^{p_d(n)}\}_\D$-transitive. And in fact we will prove a stronger result:
\begin{itemize}
  \item If $p_1, \ldots, p_d\in \mathcal{P}_0^*$ are
distinct polynomials, then $(X,T)$ is $\{T^{p_1(n)},\ldots, T^{p_d(n)}\}$-thickly-syndetic transitive, and it is $\{T^{p_1(n)},\ldots, T^{p_d(n)}\}_\D$-syndetic transitive.
\end{itemize}

\subsection{The PET-induction}\label{subsection-PET}\
\medskip

\subsubsection{}
Now $\Gamma=\Z=\langle T\rangle$, and $\PG=\{T^{p(n)}: p\in \mathcal{P}\}$. For each $T^{p(n)}\in \PG$,
its weight $w(T^{p(n)})=(1,k)$, where $k$ is the degree of $p(n)$. A system $A$ has the form of \break $\{ T^{p_1(n)}, T^{p_2(n)},
\ldots, T^{p_d(n)} \}$, where $p_1, \ldots, p_d\in \mathcal{P}$ are
distinct polynomials. Its weight vector $\phi(A)$ has the form of
\begin{equation*}
  \big(a_1(1,1), a_2 (1,2),\ldots, a_k (1,k) \big).
\end{equation*}

For example, the weight vector of $\{T^{c_1n},\ldots,T^{c_mn}\}$ is $\big(m(1,1)\big)$ if $c_1,\ldots,c_m$
are distinct and non-zero; the weight vector of $\{T^{an^2+b_1n}, \ldots, T^{an^2+b_dn}\}$ ($a\not=0$)
is $\big(1(1,2)\big)$; and the weight vector of $\{T^{an^2+b_1n}, \ldots, T^{an^2+b_dn},
T^{c_1n},\ldots, T^{c_mn}\}$ ($a\not=0$ and $c_1,\ldots,c_m$ are distinct and non-zero) is $\big( m(1,1), 1(1,2)\big)$;
and the weight vector of the general polynomials of degree $\le 2$ is $\big(m(1,1),k(1,2)\big)$.

Under the order of weight vectors, one has
\begin{equation*}
  \begin{split}
  & \big(1(1,1)\big)<\big(2(1,1)\big)<\ldots <\big(m(1,1)\big)<\ldots<\big(1(1,2)\big)<\big(1(1,1), 1(1,2)\big)<\ldots <\\
  & \big(m(1,1), 1(1,2)\big)<\ldots <\big( 2(1,2)\big)< \big(1(1,1), 2(1,2)\big)<\ldots <\big(m(1,1), 2(1,2)\big)<\ldots<\\
  &\big(m(1,1), k(1,2)\big)<\ldots < \big(1(1,3)\big)<\big(1(1,1), 1(1,3)\big)<\ldots< \big(m(1,1), k(1,2), 1(1,3)\big)\\
  &< \ldots < \big(2(1,3)\big)<\ldots< \big(a_1(1,1), a_2 (1,2),\ldots, a_k (1,k) \big)<\ldots.
  \end{split}
\end{equation*}

\subsubsection{}
To prove Theorem \ref{main-1}, we will use induction on the weight vectors.
We start from the systems with the weight vector $(1(1,1))$, i.e. $A=\{T^{a_1n}\}$.
After that, we assume that the result holds for all systems whose weight vectors are
$< \big(a_1(1,1), a_2 (1,2),\ldots, a_k (1,k) \big)$. Then we show
that the result also holds for the system with weight vector \break $\big(a_1(1,1), a_2 (1,2),\ldots, a_k (1,k) \big)$, and hence the proof is completed.

To illustrate the basic ideas, we show the result for the system $A=\{T^{n^2}, T^{2n^2}\}$, whose weight vector is $\big(2 (1, 2) \big)$.
The general proof of Theorem \ref{main-1} is similar, and we omit it here. We will give the details in the proof of Theorem \ref{stronger-thm}.

\subsection{Example: $(X,T)$ is $\{T^{n^2}, T^{2n^2}\}_\D$-transitive.}\
\medskip

To show this example, we need to verify the following cases one by one:

{\it

\medskip

\noindent {\bf Case 1} when the  weight vector is $\big(d(1,1)\big)$:
$(X,T)$ is $\{T^{a_1n},\ldots,T^{a_dn}\}_\D$-syndetic transitive, where $a_1,\ldots, a_d\in \Z\setminus\{0\}$ are distinct integers.

\medskip

\noindent {\bf Case 2} when the weight vector is $\big(1(1,2)\big)$:

\begin{enumerate}
\item $(X,T)$ is $\{T^{an^2+b_1n},\ldots,T^{an^2+b_dn}\}$-thickly-syndetic transitive,
\item $(X,T)$ is $\{T^{an^2+b_1n},\ldots,T^{an^2+b_dn}\}_\D$-syndetic transitive, where $b_1,\ldots,b_d$ are distinct integers and $a\in \Z\setminus\{0\}$.
\end{enumerate}

\medskip

\noindent {\bf Case 3} when the weight vector is $\big(r(1,1), 1(1,2)\big)$:

\begin{enumerate}
 \item $(X,T)$ is $\{T^{c_1n},\ldots,T^{c_rn},T^{an^2+b_1n},\ldots,T^{an^2+b_dn}\}$-thickly-syndetic transitive,

 \item  $(X,T)$ is $\{T^{c_1n},\ldots,T^{c_rn},T^{an^2+b_1n},\ldots,T^{an^2+b_dn}\}_\D$-syndetic transitive, where $a\in \Z\setminus\{0\}$, $b_1,\ldots,b_d$ are distinct integers and $c_1,\ldots,c_r$ are distinct non-zero integers.
\end{enumerate}

\noindent {\bf Case 4} when the weight vector is $\big(2(1,2)\big)$:
\begin{enumerate}
\item $(X,T)$ is $\{T^{n^2}, T^{2n^2}\}$-thickly-syndetic transitive,
\item $(X,T)$ is $\{T^{n^2}, T^{2n^2}\}_\D$-transitive.
\end{enumerate}}

\subsubsection{Case 1: $(X,T)$ is $\{T^{a_1n},\ldots,T^{a_dn}\}_\D$-syndetic transitive, where $a_1,\ldots, a_d$ are distinct non-zero integers.}\label{Case1}\

\begin{proof}
We will prove Case 1 by induction on $d$. By Lemma \ref{lem-hy}, Case 1 holds for $d=1$. Now we assume that the
result holds for $d\ge 1$. That is, for any non-empty open subsets $U, V_1,\ldots, V_d$ of $X$ and for distinct non-zero integers $c_1,\ldots, c_d$,
\begin{equation*}
 \{n\in \Z: U\cap T^{-c_1n}V_1 \cap\ldots\cap T^{-c_dn}V_d\neq \emptyset\}
\end{equation*}
is a syndetic set.

\medskip

Now let $U, V_1,\ldots, V_d,V_{d+1}$ be non-empty open subsets of $X$ and $a_1,\ldots, a_d, a_{d+1}$
are distinct non-zero integers. We will show that
\begin{equation*}
  N:=\{n\in \Z: U\cap T^{-a_1n}V_1 \cap\ldots\cap T^{-a_dn}V_d\cap T^{-a_{d+1}n}V_{d+1}\neq \emptyset\}
\end{equation*}
is syndetic. Write $p_1(n)=a_1n, \ldots, p_{d+1}(n)=a_{d+1}n$.

Since $(X,T)$ is minimal, there is some $\ell \in \N$ such that $X=\bigcup_{j=0}^\ell T^{j}U$.

By Corollary \ref{cor-hy}, $(X^{d+1}, T^{a_1}\times \ldots \times T^{a_{d+1}})$ is weakly mixing.
By Lemma \ref{lem-KO}, there are non-empty subsets $V_1^{(\ell)},\ldots,V_{d+1}^{(\ell)}$
and integers $k_0, k_1,\ldots,k_\ell$ such that for each $i=1,2,\ldots,d+1$, one has that
\begin{equation*}
  T^{p_i(k_j)}T^{-j}V_i^{(\ell)}\subseteq V_i,  \quad \text{for all}\quad 0\le j\le \ell.
\end{equation*}

Let $q_1(n)=p_2(n)-p_1(n)=(a_2-a_1)n,\ldots, q_d(n)=p_{d+1}(n)-p_1(n)=(a_{d+1}-a_1)n$. Since $a_2-a_1,\ldots,a_{d+1}-a_1$
are distinct non-zero integers, by the induction hypothesis,
\begin{equation*}
  E=\{n\in \Z: V_1^{(\ell)}\cap T^{-q_1(n)}V_2^{(\ell)} \cap\ldots\cap T^{-q_d(n)}V_{d+1}^{(\ell)}\neq \emptyset \}
\end{equation*}
is syndetic.

Let $m\in E$. Then there is some $x_m\in V_1^{(\ell)}$ such that $T^{q_i(m)}x_m\in V_{i+1}^{(\ell)}$
for $i=1,\ldots,d$. Clearly, there is some $y_m\in X$ with $y_m=T^{-p_1(m)}x$. Since $X=\bigcup_{j=0}^\ell T^{j}U$,
there is some $b_m\in \{0,1,\ldots, \ell\}$ such that $T^{b_m}z_m=y_m$ for some $z_m\in U$. Thus for each $i=1,2,\ldots, d+1$
\begin{equation*}
  \begin{split}
  T^{p_i(m+k_{b_m})}z_m& =T^{p_i(m+k_{b_m})} T^{-b_m}y_m=T^{p_i(m+k_{b_m})} T^{-b_m}T^{-p_1(m)}x_m\\
  &= T^{p_i(k_{b_m})}T^{-b_m}T^{p_i(m)-p_1(m)}x_m\\
  &=T^{p_i(k_{b_m})}T^{-b_m}T^{q_{i-1}(m)}x_m \quad (\text{Let $q_0(n)=0$})\\
  & \in T^{p_i(k_{b_m})}T^{-b_m} V_{i}^{(\ell)}\subseteq V_{i}.
  \end{split}
\end{equation*}
That is,
$$z_m\in U\cap T^{-p_1(n)}V_1 \cap\ldots\cap T^{-p_d(n)}V_d\cap T^{-p_{d+1}(n)}V_{d+1},$$
where $n=m+k_{b_m}$. Thus
\begin{equation*}
  N\supseteq \{m+k_{b_m}:m\in E\}
\end{equation*}
is a syndetic set. By induction the proof is completed.
\end{proof}

\begin{rem}
Note that Case {\red 1} is the strengthened version of Glasner's theorem.
\end{rem}

\subsubsection{Case 2: (1) $(X,T)$ is $\{T^{an^2+b_1n},\ldots,T^{an^2+b_dn}\}$-thickly-syndetic transitive,
where $b_1,\ldots,b_d$ are distinct integers and $a\in \Z\setminus\{0\}$}\label{Case2}
\

\begin{proof}
Since the family of thickly-syndetic sets is a filter, it suffices to show that for any $p(n)=an^2+bn$ ($a\neq 0$, $a,b\in \Z$), one has that for all non-empty open sets $U,V\subseteq X$
\begin{equation*}
  N_p(U,V)=\{n\in \Z: U\cap T^{-p(n)}V\neq \emptyset \}
\end{equation*}
is thickly-syndetic.

Since $(X,T)$ is minimal, there is some $\ell\in \N$ such that $X=\bigcup_{i=0}^\ell T^{i}U$.

Let $L\in \N$ and let $k_i=i(L+2)$ for all $i\in \{0,1,\ldots, \ell\}$. Since $(X,T)$ is
weakly mixing and minimal, by Lemma \ref{lem-hy}
$$C:=\bigcap_{{(i,j)\in {\{0,1,\ldots,\ell\}\times \{0,1,\ldots,L\}}}} \{ k\in \mathbb{Z}:
V\cap T^{-k}\big( T^{p(k_i+j)-i})^{-1}V\big)\neq \emptyset\}$$
is a thickly-syndetic set. Choose $c\in C$. Then for any $(i,j)\in \{0,1,\ldots,\ell\}\times \{0,1,\ldots,L\}$ one has
$$V_{i,j}:=V\cap (T^{p(k_i+j)+c-i})^{-1}V$$
is a non-empty open subset of $V$ and
$$T^{p(k_i+j)+c-i}V_{i,j}\subset V.$$

Let $p_{i,j}(n)=p(k_i+j+n)-p(k_i+j)-p(n)=2ak_in+2ajn$ for any $(i,j)\in \{0,1,\ldots,\ell\}\times \{0,1,\ldots,L\}$.
Since $k_i=i(L+2)$, $p_{i,j}$ are distinct for all $(i,j)\in \{0,1,\ldots,\ell\}\times \{0,1,\ldots,L\}$.
By Case {\red 1},
$$D:=\{n \in \mathbb{Z}: V\cap \bigcap_{(i,j)\in \{0,1,\ldots,\ell\}\times \{0,1,\ldots,L\}} T^{-p_{i,j}(n)}V_{i,j}\neq \emptyset\}$$
is a syndetic set.

For $m\in D$, there exists $x_m\in V$ such that $T^{p_{i,j}(m)}x_m\in V_{i,j}$
for any $(i,j)\in \{0,1,\ldots,\ell\}\times \{0,1,\ldots,L\}$.
Let $y_m=T^{-p(m)}x_m$. Since $X=\bigcup_{i=0}^{\ell}T^{i}U$, there are $z_m\in U$ and $0\le b_m\le \ell$ such that $T^{c}y_m=T^{b_m}z_m$.
Then $z_m=T^{-p(m)+c-b_m}x_m$ and we have
\begin{equation*}
\begin{split}
T^{p(m+k_{b_m}+j)}z_m&=T^{p(m+k_{b_m}+j)}T^{-p(m)+c-b_m}x_m\\
&  = T^{p(k_{b_m}+j)+c-b_m}\big(T^{p(k_{b_m}+j+m)-p(k_{b_m}+j)-p(m)}x_m\big)\\
&=T^{p(k_{b_m}+j)+c-b_m}(T^{p_{b_m,j}(m)}x_m)\\
& \in T^{p(k_{b_m}+j)+c-b_m}V_{b_m,j}\subset  V
\end{split}
\end{equation*}
for each for $j\in \{0,1,\ldots,L\}$.
Thus $$\{ m+k_{b_m}+j:0\le j\le L\} \subset N_p(U,V).$$
Hence the set $\{ n\in \mathbb{Z}: n+j\in N_p(U,V) \text{ for any }j\in \{0,1,\ldots,L\}\}$
contains the syndetic set $\{m+k_{b_m}:m\in D\}$.
As $L\in \mathbb{N}$ is arbitrary, $N_p(U,V)$ is a thickly-syndetic set.
\end{proof}

\subsubsection{Case 2: (2) $(X,T)$ is $\{T^{an^2+b_1n},\ldots,T^{an^2+b_dn}\}_\D$-syndetic transitive, where $b_1,\ldots,b_d$ are distinct integers and $a\in \Z \setminus \{0\}$. } \label{Case3}
\

\begin{proof}
Let $p_1(n)=an^2+b_1n,\ldots,p_d(n)=an^2+b_dn$. We will show for any given non-empty open subsets $U, V_1,\ldots, V_d$ of $X$
$$N=\{n\in \Z: U\cap (T^{-p_1(n)}V_1\cap  \ldots \cap T^{-p_d(n)}V_d)\not=\emptyset\}$$
is syndetic.

Since $(X,T)$ is minimal, there is some $\ell\in \N$ such that $X=\bigcup_{i=0}^\ell T^{i}U$.
Then by Case {\red 2}(1) and Lemma \ref{freedom} there are integers $\{k_j\}_{j=0}^{\ell}$ and non-empty open sets $V_i^{(\ell)}\subset V_i$, $1\le i\le d$ such that
$|k_{j}|>|k_{j-1}|+\sum_{i=1}^d |b_i|$ for $j=0,\cdots,\ell$ ($k_{-1}=0$)  and
$$T^{p_i(k_j)}T^{-j}V_i^{(\ell)}\subset V_i,\ \  0\le j\le \ell, 1\le i\le d.$$
Let $q_i(m,n)=p_i(n+m)-p_i(m)-p_1(n)$ for $n,m\in \mathbb{Z}$ and $i=1,\cdots,d$.
Since $|k_{j}|>|k_{j-1}|+\sum_{i=1}^d |b_i|$ for $j=0,\cdots,\ell$, we have that all $q_i(k_j,n)$ are distinct polynomials in $n$ with degree 1 for
$0\le j\le \ell, 1\le i\le d$.
By Case {\red 1},
\begin{equation*}
  E=\{n\in \Z: V_1^{(\ell)}\cap \bigcap_{j=1}^\ell\Big( T^{-q_1(k_j,n)}V_1^{(\ell)} \cap\ldots\cap T^{-q_d(k_j,n)}V_{d}^{(\ell)}\Big)\neq \emptyset \}
\end{equation*}
is syndetic.

Let $m\in E$. Then there is some $x_m\in V_1^{(\ell)}$ such that
$$T^{q_i(k_j,m)}x_m\in V_{i}^{(\ell)}\ \text{for all $1\le i\le d$ and $0\le j\le \ell$.}$$
Clearly, there is some $y_m\in X$ such that $y_m=T^{-p_1(m)}x$. Since $X=\bigcup_{j=0}^\ell T^{j}U$, there is some $b_m\in \{0,1,\ldots, \ell\}$ such that $T^{b_m}z_m=y_m$ for some $z_m\in U$. Thus for each $i=1,2,\ldots, d$
\begin{equation*}
  \begin{split}
  T^{p_i(m+k_{b_m})}z_m& =T^{p_i(m+k_{b_m})} T^{-b_m}y_m=T^{p_i(m+k_{b_m})} T^{-b_m}T^{-p_1(m)}x_m\\
  &= T^{p_i(k_{b_m})}T^{-b_m}T^{p_i(m+k_{b_m})-p_i(k_{b_m})-p_1(m)}x_m\\
  &=T^{p_i(k_{b_m})}T^{-b_m}T^{q_{i}(k_{b_m}, m)}x_m \\
  & \in T^{p_i(k_{b_m})}T^{-b_m} V_{i}^{(\ell)}\subseteq V_{i}.
  \end{split}
\end{equation*}
That is,
$$z_m\in U\cap T^{-p_1(n)}V_1 \cap\ldots\cap T^{-p_d(n)}V_d,$$
where $n=m+k_{b_m}$. Thus
\begin{equation*}
  N\supseteq \{m+k_{b_m}:m\in E\}
\end{equation*}
is a syndetic set. The proof is completed.
\end{proof}

\subsubsection{Case 3: (1) $(X,T)$ is
$\{T^{c_1n},\ldots,T^{c_rn},T^{an^2+b_1n},\ldots,T^{an^2+b_dn}\}$-thickly-syndetic transitive, where $a\in \Z\setminus\{0\}$, $b_1,\ldots,b_d$ are distinct integers and $c_1,\ldots,c_r$ are distinct non-zero integers.}\label{Case4}
\

\begin{proof}
It follows from Case {\red 2}(1) and Lemma \ref{lem-hy}.
\end{proof}

\subsubsection{Case 3: (2) $(X,T)$ is $\{T^{c_1n},\ldots,T^{c_rn},T^{an^2+b_1n},\ldots,T^{an^2+b_dn}\}_\D$-syndetic transitive, where $a\in \Z\setminus\{0\}$, $b_1,\ldots,b_d$ are distinct integers and $c_1,\ldots,c_r$ are distinct non-zero integers.}\label{Case5}
\

\begin{proof}
The proof is almost the same to the proof of Case {\red 2} (2). The only difference is that we need to deal with it by induction on $r$.

Let $p_1(n)=c_1n,\ldots,p_r(n)=c_rn$, $p_{r+1}(n)=an^2+b_1n,\ldots,p_{r+d}(n)=an^2+b_dn$. We will show for any given non-empty open sets $U, V_1,\ldots, V_t$ (where $t=r+d$)
$$N=\{n\in \Z: U\cap (T^{-p_1(n)}V_1\cap  \ldots \cap T^{-p_t(n)}V_t)\not=\emptyset\}$$
is syndetic.

Since $(X,T)$ is minimal, there is some $\ell\in \N$ such that $X=\bigcup_{i=0}^\ell T^{i}U$.
Then by Case {\red 3}(1) and Lemma \ref{freedom} there are integers $\{k_j\}_{j=0}^{\ell}$ and non-empty open sets $V_i^{(\ell)}\subset V_i$, $1\le i\le t$ such that
$|k_{j}|>|k_{j-1}|+\sum_{i=1}^d |b_i|$ for $j=0,\cdots,\ell$ (here $k_{-1}=0$)  and
$$T^{p_i(k_j)}T^{-j}V_i^{(\ell)}\subset V_i,\ \  0\le j\le \ell, 1\le i\le t.$$

Let $q_i(m,n)=p_i(n+m)-p_i(m)-p_1(n)$ for $n,m\in \mathbb{Z}$ and $i=1,\cdots,t$.
Then $$q_i(k_j,n)=(c_i-c_1)n$$ for $n\in \mathbb{Z}$,  $1\le i\le r$ and $0\le j\le \ell$. Since $|k_{j}|>|k_{j-1}|+\sum_{i=1}^d |b_i|$ for $j=0,\cdots,\ell$,
we have that all $q_{r+i}(k_j,n)=an^2+(2ak_j+b_i-c_1)n$ are distinct polynomials in $n$ with degree 2 for $0\le j\le \ell, 1\le i\le d$.

Note that $q_1(k_j,n)=0$ for $n\in \mathbb{Z}$ and $0\le j\le \ell$.
By Case {\red 2}(2) if $r=1$, or by the inductive assumption if $r\ge 2$,
\begin{align*}
  E&=\{n\in \Z: V_1^{(\ell)}\cap \bigcap_{j=1}^\ell\Big( T^{-q_1(k_j,n)}V_1^{(\ell)} \cap\ldots\cap T^{-q_t(k_j,n)}V_{t}^{(\ell)}\Big)\neq \emptyset \}\\
&=\{n\in \Z: V_1^{(\ell)}\cap \Big( \bigcap_{i=2}^r T^{-(c_i-c_1)n}V_i^{(\ell)} \Big) \cap  \bigcap_{j=1}^\ell \Big( \bigcap_{i=1}^d T^{-q_{r+i}(k_j,n)}V_{r+i}^{(\ell)}\Big)\neq \emptyset \}
\end{align*}
is syndetic.
The rest of proof is the same to the proof in Case {\red 2}(2).
\end{proof}

\subsubsection{Case 4: (1) $(X,T)$ is $\{T^{n^2}, T^{2n^2}\}$-thickly-syndetic transitive.}\label{Case6}
\

\begin{proof}
It follows from Case {\red 2}(1).
\end{proof}

\subsubsection{Case 4: (2) $(X,T)$ is $\{T^{n^2}, T^{2n^2}\}_\D$-transitive.}\label{Case7}
\

\begin{proof}
The proof is almost the same to the proof of Case {\red 2}(2).
Let $p_1(n)=n^2$, $p_2(n)=2n^2$. We will show for any given non-empty open subsets $U, V_1, V_2$ of $X$
$$N=\{n\in \Z: U\cap (T^{-p_1(n)}V_1\cap  T^{-p_2(n)}V_2)\not=\emptyset\}$$
is syndetic.

Since $(X,T)$ is minimal, there is some $\ell\in \N$ such that $X=\bigcup_{i=0}^\ell T^{i}U$.
Then by Case {\red 4}(1) and Lemma \ref{freedom} there are integers $\{k_j\}_{j=0}^{\ell}$ and non-empty open sets $V_i^{(\ell)}\subset V_i$, $1\le i\le \ell$ such that
$|k_{j}|>|k_{j-1}|$ for $j=0,\cdots,\ell$ (here $k_{-1}=0$) and
$$T^{p_i(k_j)}T^{-j}V_i^{(\ell)}\subset V_i,\ \  1\le j\le \ell, 1\le i\le 2.$$

Let $q_i(m,n)=p_i(n+m)-p_i(m)-p_1(n)$  for $n,m\in \mathbb{Z}$ and $i=1,2$. Since $\{|k_j|\}$ is an increasing sequence of natural numbers, we have that
all $$q_i(k_j,n)=\begin{cases} 2k_jn \, &\text{ if }i=1 \\n^2+2k_jn \, &\text{ if } i=2 \end{cases}$$
are distinct non-constant polynomials in $n$ for $0\le j\le \ell, 1\le i\le 2$.
By Case {\red 3}(2),
\begin{equation*}
  E=\{n\in \Z: V_1^{(\ell)}\cap \bigcap_{j=1}^\ell\Big( T^{-q_1(k_j,n)}V_1^{(\ell)} \cap T^{-q_2(k_j,n)}V_{2}^{(\ell)}\Big)\neq \emptyset \}
\end{equation*}
is syndetic. The same proof to the Case {\red 2}(2), for all $m\in E$ one finds some $b_m\in \{0,\ldots, \ell\}$ such that $m+k_{b_m}\in N$, and hence
\begin{equation*}
  N\supseteq \{m+k_{b_m}:m\in E\}
\end{equation*}
is a syndetic set. The proof is completed.
\end{proof}

\section{Proof of Theorems \ref{stronger-thm} and \ref{ye00}}\label{section-proof of Main}

In this section, we give a proof of Theorems \ref{stronger-thm} and \ref{ye00}.

\subsection{}

Let $(X,\Gamma)$ be a t.d.s., where $\Gamma$ is a nilpotent group such that for each $T\in \Gamma$, $T\neq e_\Gamma$,
is weakly mixing and minimal. Thus, $\Gamma$ is a nilpotent group without torsion.
For $d,k\in \N$ let $T_1,\ldots,T_d\in \Gamma$, and $p_{i,j}\in \mathcal{P}_0, 1\le i\le k, 1\le j\le d$
such that the expressions
\begin{equation*}
  g_i(n)=T_1^{p_{i,1}(n)}\cdots T_d^{p_{i,d}(n)}
\end{equation*}
depends nontrivially on $n$ for $i=1,2,\cdots,k$, and for all $i\neq j\in \{1,2,\ldots,k\}$ the expressions $g_i(n)g_j(n)^{-1}$
depend nontrivially on $n$.

By Lemma \ref{diagonal1}, to prove Theorem \ref{stronger-thm} it remains to show that
for any given non-empty open sets $U, V_1,\ldots, V_k$ of $X$
there is $n\in\mathbb{N}$ such that
\begin{equation*}
  U\cap (g_1(n)^{-1}V_1\cap \ldots \cap g_k(n)^{-1}V_k)\not=\emptyset,
\end{equation*}
i.e. $(X,\Gamma)$ is {\em $A_\D$-transitive}. Moreover, we also need to show
it is $A$-thickly-syndetic transitive in the same time, where $A=\{g_1,\ldots,g_k\}$.






\subsection{Some lemmas}\
\medskip

\subsubsection{Some basic results by Leibman}

\begin{lem}\label{lem-gamma}\cite[Lemma 2.4.]{Leibman94}
Let $g$ be a $\Gamma$-polynomial.
\begin{enumerate}
  \item If $h$ is a $\Gamma$-polynomial and $g'=h^{-1}gh$, then $g'\sim g$.
  \item If $m\in \N$ and $g'$ is defined by $g'(n)=g^{-1}(m)g(n+m)$, then $g'\sim g$.
  \item \begin{enumerate}
          \item If $g', h$ are $\Gamma$-polynomials such that $g'\sim g$, $h\not\sim g $ and $w(h)\preccurlyeq w(g)$,
          then $g'h^{-1}\sim gh^{-1}$ and $w(gh^{-1})=w(g)$
          \item If $h\neq e_\Gamma$ is a $\Gamma$-polynomial such that $h\sim g$, then $w(gh^{-1}) \prec w(g)$.
        \end{enumerate}
\end{enumerate}
\end{lem}

\begin{cor}\cite[Corollary 2.5.]{Leibman94}\label{leibman}
Let $A$ be a system.
\begin{enumerate}
  \item If $A'$ is a system consisting of $\Gamma$-polynomials of the form $g' = h^{-1}gh$
for $g\in A$ and $h$ being a $\Gamma$-polynomial, then $\phi(A')\preccurlyeq \phi(A)$.
  \item If $A'$ is a system consisting of $\Gamma$-polynomials $g'$ satisfying the equality
$g'(n)=g^{-1}(m)g(n+m)$ for some $g\in A$ and some $m\in \N$, then $\phi(A')\preccurlyeq \phi(A)$.
  \item Let $h\in A, h\neq e_\Gamma$, be a $\Gamma$-polynomial of weight minimal in $A$:
$w(h)\le w(g)$ for any $g\in A$. If $A'$ is a system consisting of $\Gamma$-polynomials
of the form $g'=gh^{-1}, g\in A$, then $\phi(A')\prec \phi(A)$.
\end{enumerate}
\end{cor}

\subsubsection{Additional lemmas}\label{add-1}
To show the main result we find that above lemma and corollary are not enough. We need some additional lemmas
which we shall prove in this subsection.

Using \eqref{eq-111} and Theorem \ref{thm-5.1}(2), it is clear that for $\Gamma$-polynomial $g$, if
$$\{n\in \mathbb{Z}: g(n)=e_{\Gamma}\}$$
is an infinite set then $g \equiv e_{\Gamma}$ since every non-zero integral polynomial has finitely many zero points.
In fact if $|\{n\in \mathbb{Z}: g(n)=e_{\Gamma}\}|>k$ for some $k$ depending only on $g$, then $g \equiv e_{\Gamma}$.

\begin{lem}\label{lem-equi1} Let $f,g\in \PG_0$. Then
\begin{enumerate}
\item If $\{ k'\in \mathbb{Z}: f(k')^{-1}f(n+k')=g(n) \text{ for all }n\in \mathbb{Z}\}$
is an infinite set, then $g(n)=f(n)=(f(1))^n$ for all $n\in \mathbb{Z}$ .

\item If $\{ k'\in \mathbb{Z}: f(k')^{-1}f(n+k')=g(k')^{-1}g(n+k') \text{ for all }n\in \mathbb{Z}\}$
is an infinite set, then $g=f$.
\end{enumerate}
\end{lem}
\begin{proof} 1). Let $E=\{ k'\in \mathbb{Z}: f(k')^{-1}f(n+k')=g(n) \text{ for all }n\in \mathbb{Z}\}$.
For $n\in \mathbb{Z}$, we consider the $\Gamma$-polynomial $p_n(k)=f(k)^{-1}f(n+k)g(n)^{-1}$ with respect to $k$.
Since $E\subseteq \{ k\in \mathbb{Z}: p_n(k)=e_{\Gamma}\}$, one has $p_n(k)= e_{\Gamma}$ for all $k\in \mathbb{Z}$. This implies
$f(k)^{-1}f(n+k)=g(n)$ for all $n,k\in \mathbb{N}$. Note that $f(0)=e_\Gamma$, so $f(n)=g(n)$ for all $n\in \Z$.
Then it follows from the equation $f(n+k)=f(n)f(k)$ for all $n,k\in \Z$, one has that
$$f(n)=g(n)=f(1)^n,\ \text{for all $n\in \mathbb{Z}$}.$$

2). Let $F=\{ k'\in \mathbb{Z}: f(k')^{-1}f(n+k')=g(k')^{-1}g(n+k') \text{ for all }n\in \mathbb{Z}\}$.
For $n\in \mathbb{Z}$, we consider the polynomial $q_n(k)=f(k)^{-1}f(n+k)(g(k)^{-1}g(n+k))^{-1}$ with respect to $k$.
Since $F\subseteq \{ k\in \mathbb{Z}: q_n(k)=e_{\Gamma}\}$, one has $q_n(k)= e_{\Gamma}$ for all $k\in \mathbb{Z}$. This implies
$f(k)^{-1}f(n+k)=g(k)^{-1}g(n+k)$ for all $n,k\in \mathbb{N}$. Since $f(0)=e_\Gamma$ and $g(0)=e_\Gamma$, one has that $f(n)=g(n)$ for all $n\in \mathbb{Z}$.
\end{proof}

\begin{lem}\label{lem-diff}
Let $f\in \PG_0$. If for each $m\in \mathbb{Z}\setminus \{0\}$ there is  some $n=n(m)\in \mathbb{Z}$ such that $f(m+n)\neq f(m)f(n)$, then for any $\ell,L\in \mathbb{N}$ we can find $k_0,k_1,\cdots,k_\ell\in \mathbb{N}$ such that
\begin{enumerate}
\item $f(k_i+j)^{-1}f(k_i+j+n)f(n)^{-1}\in \PG_0^*$ for any $(i,j)\in \{0,1,\ldots,\ell\}\times \{0,1,\ldots,L\}$.

\item $f(k_i+j)^{-1}f(k_i+j+n)f(n)^{-1}$ and $f(k_{i'}+j')^{-1}f(k_{i'}+j'+n)f(n)^{-1}$ are  distinct
$\Gamma$-polynomials with respect to $n$
for any $(i,j),(i',j')\in \{0,1,\ldots,\ell\}\times \{0,1,\ldots,L\}$ with $(i,j)\neq (i',j')$.
\end{enumerate}
\end{lem}

\begin{proof}
Let
$$K=\{k\in\mathbb{Z}: f(k)^{-1}f(n+k)=f(n) \text{ for all }n\in \mathbb{Z}\}.$$
Then by
Lemma \ref{lem-equi1}(1), $K$ is a finite set. Fix $\ell,L\in\mathbb{N}$.
For $(i,j),(i',j')\in \{0,1,\ldots,\ell\}\times \{0,1,\ldots,L\}$ with $(i,j)\neq (i',j')$,
we denote $E((i,j),(i',j'))$ by the set of all $k\in \mathbb{Z}$ satisfying that
$$f(i(L+2)+k+j)^{-1}f(i(L+2)+k+j+n)=f(i'(L+2)+k+j')^{-1}f(i'(L+2)+k+j'+n)$$
for all $n\in \mathbb{Z}$.

We claim that $E((i,j),(i',j'))$ is a finite set. Assume the contrary that $E((i,j),(i',j'))$ is not a finite set.
Let $m=(i'-i)(L+2)+(j'-j)$. Since $(i,j)\neq (i',j')$, one has $m\neq 0$. Put  $g(n)=f(m)^{-1}f(n+m)$ for $n\in \mathbb{Z}$.
Then $g\in \PG_0$ and $g\neq f$ by the assumption of the lemma. Note that
$$g(i(L+2)+k+j)^{-1}g(i(L+2)+k+j+n)=f(i'(L+2)+k+j')^{-1}f(i'(L+2)+k+j'+n)$$
for all $n\in \mathbb{Z}$.  We have
\begin{align*}
& \{ k\in \mathbb{Z}:f(k)^{-1}f(k+n)=g(k)^{-1}g(n+k)\text{ for all }n\in \mathbb{Z}\}\\
&=E((i,j),(i',j'))+(i(L+2)+j)
\end{align*}
is an infinite set. Thus by Lemma \ref{lem-equi1}(2) one has that $f=g$, a contradiction! This shows that
$E((i,j),(i',j'))$ is a finite set.

\medskip

Set 
\begin{equation*}
 E=\bigcup_{ {(i,j)\neq (i',j')\atop{\in \{0,1,\ldots,\ell\}\times \{0,1,\ldots,L\}}}}  E((i,j),(i',j')).
\end{equation*}
Then $E$ is also a finite set. Put $$F=E\cup \{k-(i(L+2)+j): k\in K,(i,j)\in \{0,1,\ldots,\ell\}\times \{0,1,\ldots,L\}\}.$$
It is clear that $F$ is finite.

Now we take $u\in \mathbb{N}\setminus F$.  Let $k_i=i(L+2)+u$ for $i\in \{0,1,\ldots,\ell\}$.
On one hand, for any $(i,j)\in \{0,1,\ldots,\ell\}\times \{0,1,\ldots,L\}$ one has $f(k_i+j)^{-1}f(k_i+j+n)f(n)^{-1}\in \PG_0^*$ as $k_i+j\not \in K$. On the other hand, for $(i,j),(i',j')\in \{0,1,\ldots,\ell\}\times \{0,1,\ldots,L\}$ with $(i,j)\neq (i',j')$ one has that
$f(k_i+j)^{-1}f(k_i+j+n)f(n)^{-1}$ and $f(k_{i'}+j')^{-1}f(k_{i'}+j'+n)f(n)^{-1}$ are distinct $\Gamma$-polynomials with respect to $n$ since $u\not \in
E((i,j),(i',j'))$. Thus we finish the proof of the lemma.
\end{proof}

\begin{lem}\label{lem-equi2}
Let $f_1,f_2,\ldots,f_v\in \PG_0^*$ be distinct $\Gamma$-polynomials. Then there exists a sequence $\{r(i)\}_{i=0}^\infty$ of natural numbers such that
for any $\ell\in \mathbb{N}$ and $k_0,k_1,\cdots,k_\ell\in \mathbb{N}$ with $k_0>r(0)$ and $k_i>k_{i-1}+r(k_{i-1})$ for $i=1,\cdots, \ell$, one has that
\begin{enumerate}
\item $f_t(k_i)^{-1}f_t(n+k_i)f_1(n)^{-1}\in \PG_0^*$ for any $t\in  \{2,\cdots,v\}$ and $i\in \{0,1,\ldots,\ell\}$.
\item  $f_t(k_i)^{-1}f_t(n+k_i)f_1(n)^{-1}$ and $f_s(k_j)^{-1}f_s(n+k_j)f_1(n)^{-1}$ are distinct $\Gamma$-polynomial with respect to $n$ for
any $t\neq s \in \{1,2,\cdots,v\}$ and $i,j\in  \{0,1,\ldots,\ell\}$.
\end{enumerate}
\end{lem}
\begin{proof} First $f_t(k)^{-1}f_t(n+k)f_1(n)^{-1}\in \PG_0$ for any $t\in  \{1,2,\cdots,v\}$ and $k\in \mathbb{Z}$.
Then by Lemma \ref{lem-equi1} (1), for  $t\in \{2,\ldots,v\}$
$$K_t:=\{ k\in \mathbb{Z}: f_t(k)^{-1}f_t(n+k)=f_1(n) \text{ for all }n\in \mathbb{Z}\}$$
is a finite set since $f_1\neq f_t$. Thus for  $t\in \{2,\ldots,v\}$ we may take $L_{t}\in \mathbb{N}$ such that
$K_t\subseteq [-L_t,L_t]$.

\medskip
For any  $t, s \in \{1,2,\cdots,v\}$ and $k'\in \mathbb{Z}$, we put
$$K_{t,s}(k'):=\{ k\in \mathbb{Z}: f_t(k)^{-1}f_t(n+k)=f_s(k')^{-1}f_s(n+k') \text{ for all }n\in \mathbb{Z}\}.$$
If $K_{t,s}(k')$ is an infinite set then by Lemma \ref{lem-equi1} (1) one has
$$f_s(k')^{-1}f_s(n+k')=f_t(n)=(f_t(1))^n$$
 for all $n\in \mathbb{Z}$. Take $n=-k'$ one has
$f_s(k')=(f_t(1))^{k'}$ as $f_s(0)=e_\Gamma$. Thus
$$f_s(m)=f_s(k')f_t(m-k')=(f_t(1))^{k'}(f_t(1))^{m-k'}=(f_t(1))^m=f_t(m)$$
for all $m\in \mathbb{Z}$. Hence $f_s=f_t$. This implies  $s=t$.

The above discussion shows that
$K_{t,s}(k')$ is a finite set for any  $t\neq s \in \{1,2,\cdots,v\}$ and $k'\in \mathbb{Z}$.
Thus for any  $t\neq s \in \{1,2,\cdots,v\}$ and $k'\in \mathbb{Z}$ we may take $L_{t,s}(k')\in \mathbb{N}$ such that
$K_{t,s}(k')\subseteq [-L_{t,s}(k'),L_{t,s}(k')]$.

Next by Lemma \ref{lem-equi1} (2), for  $t\neq s \in \{1,2,\cdots,v\}$
$$K_{t,s}:=\{ k\in \mathbb{Z}: f_t(k)^{-1}f_t(n+k)=f_s(k)^{-1}f_s(n+k) \text{ for all }n\in \mathbb{Z}\}$$
is a finite set since $f_t\neq f_s$. Thus for  $t\neq s \in \{1,2,\cdots,v\}$ we may take $L_{t,s}\in \mathbb{N}$ such that
$K_{t,s}\subseteq [-L_{t,s},L_{t,s}]$.

For $i\ge 0$, we take
$$r(i)=1+\max_{t\in \{2,\cdots,v\}} L_t+\max_{{t\neq s\in \{1,\ldots,v\}, \atop{k'\in \{0,1,\ldots,i\}}}}(L_{t,s}+L_{t,s}(k')).$$
Now for any given $\ell\in \mathbb{N}$ and $k_0,k_1,\cdots,k_\ell\in \mathbb{N}$ with $k_0>r(0)$ and $k_i>k_{i-1}+r(k_{i-1})$ for $i=1,\cdots, \ell$,
on the one hand for any $t\in  \{2,\cdots,v\}$ and $i\in \{0,1,\ldots,\ell\}$ one has $f_t(k_i)^{-1}f_t(n+k_i)f_1(n)^{-1}\in \PG_0^*$
since $k_i\not \in K_t$. One the other hand   for any $t\neq s \in \{1,2,\cdots,v\}$ and $0\le i\le j\le \ell$ one has $f_t(k_i)^{-1}f_t(n+k_i)$ and $f_s(k_j)^{-1}f_s(n+k_j)$ are distinct $\Gamma$-polynomials with respect to $n$ as $k_j\not \in K_{t,s}$ and $k_j\not \in \bigcup_{0\le r\le j-1} K_{t,s}(k_{r})$.
This clearly implies that $f_t(k_i)^{-1}f_t(n+k_i)f_1(n)^{-1}$ and $f_s(k_j)^{-1}f_s(n+k_j)f_1(n)^{-1}$ are distinct.
We finish the proof of the lemma.
\end{proof}

\subsection{Proof of Theorem \ref{stronger-thm}}\
\medskip

We will prove Theorem \ref{stronger-thm} using the PET-induction introduced in
Section \ref{section-nil}. We will use the notations in Section \ref{section-nil} freely.
Recall that $A=\{g_1,\ldots,g_k\}$.

\medskip

We start with the system whose weight vector is $\{1(1,1)\}$. That is, $A=\{S_1^{c_1n}\}$,
where $c_1\in \Z\setminus \{0\}$. Since $S_1^{c_1}\neq e_\Gamma$, $(X, S_1^{c_1})$ is weakly mixing and minimal. By Lemma \ref{lem-hy}
for any non-empty open sets $U_1$ and $V_1$ of $X$,
$$\{ n\in \mathbb{Z}: U_1\cap S_1^{-c_1n}V_1\neq \emptyset\}$$
is a thickly-syndetic set.
Hence $X$ is  $A$-thickly-syndetic transitive and $A_\D$-syndetic transitive.

\medskip

Now let $A\subseteq \PG_0^*$ be a system whose weight vector is greater than $\{1(1,1)\}$, and assume that for all systems $A'$ preceding $A$, we have $(X,\Gamma)$ is $A'$-thickly-syndetic transitive and $A'_\D$-syndetic transitive. Now we show that $(X,\Gamma)$ is $A$-thickly-syndetic transitive and $A_\D$-syndetic transitive.

\subsubsection{Claim: $(X,\Gamma)$ is $A$-thickly-syndetic transitive}\
\medskip

Since the intersection set of two thickly-syndetic subsets is still a thickly-syndetic subset, it is sufficient to show that for any $f\in A$, and for any given non-empty open subsets $U,V$ of $X$,
$$N_f(U,V):=\{n\in \mathbb{Z}: U\cap f(n)^{-1}V\neq \emptyset\}$$
is a thickly-syndetic set.

Let $T\in \Gamma$ be an element in the center of $\Gamma$ with $T\neq e_\Gamma$.
As $(X,T)$ is minimal there is $\ell\in\mathbb{N}$ such that $X=\bigcup_{i=0}^{\ell}T^{i}U$.
For given $f\in A$ and non-empty open subsets $U,V$ of $X$, we have the following two cases.

\medskip

\noindent{\bf Case 1}:
The first case is that there exists $m\in \mathbb{Z}\setminus \{0\}$ such that
$f(n+m)=f(m)f(n)$ for all $n\in \mathbb{Z}$. Thus  $f(-m)=(f(m))^{-1}$ and
$f(n-m)=f(-m)f(n)$ for all $n\in \mathbb{Z}$. Let $u=|m|$. Then $u>0$ and
$$f(n+u)=f(u)f(n) \text{ for all }n\in \mathbb{Z}.$$
This implies that for $r=0,1,\cdots,u-1$
$$f(ku+r)=(f(u))^k f(r) \text{ for all }k\in \mathbb{Z}.$$
Since $f\not \equiv e_\Gamma$ and $f(0)=e_\Gamma$,
one has that $f(u)\neq e_\Gamma$. Thus $(X,f(u))$ is weakly mixing and minimal.
Hence
$$B_r:=\{ k\in \mathbb{Z}: f(r)U\cap (f(u))^{-k}V\neq \emptyset\}$$
is a thickly-syndetic subset of $\mathbb{Z}$.

Put $B=\bigcap_{0\le r\le u-1} B_r$. Then $B$ is a thickly-syndetic subset of $\mathbb{Z}$.
Note that
$$N_f(U,V)\supseteq \bigcup_{0\le r\le u-1} \{ku+r:k\in B_r\}\supseteq\big \{ ku+r:k\in B, r\in \{0,1,\ldots,u-1\}\big\}.$$
Thus $N_f(U,V)$ is  a thickly-syndetic subset of $\mathbb{Z}$ as $B$ is a thickly-syndetic subset of $\mathbb{Z}$.

\medskip
\noindent{\bf Case 2}:
The second case is that for each $m\in \mathbb{Z}\setminus \{0\}$ there is  some $n=n(m)\in \mathbb{Z}$ such that $f(m+n)\neq f(m)f(n)$.
Fix $L\in \mathbb{N}$.  By Lemma \ref{lem-diff} for any $\ell,L\in \mathbb{N}$ we can find $k_0,k_1,\cdots,k_\ell\in \mathbb{N}$ such that
\begin{enumerate}
\item $f(k_i+j)^{-1}f(k_i+j+n)f(n)^{-1}\in \PG_0^*$ for any $(i,j)\in \{0,1,\ldots,\ell\}\times \{0,1,\ldots,L\}$.

\item $f(k_i+j)^{-1}f(k_i+j+n)f(n)^{-1}$ and $f(k_{i'}+j')^{-1}f(k_{i'}+j'+n)f(n)^{-1}$ are distinct $\Gamma$-polynomials with respect to $n$
for any $(i,j),(i',j')\in \{0,1,\ldots,\ell\}\times \{0,1,\ldots,L\}$ with $(i,j)\neq (i',j')$.
\end{enumerate}

Since $(X,T)$ is weakly mixing and minimal,
$$C:=\bigcap_{{(i,j)\in {\{0,1,\ldots,\ell\}\times \{0,1,\ldots,L\}}}} \{ k\in \mathbb{Z}: V\cap T^{-k}\big( (f(k_i+j)T^{-i})^{-1}V\big)\neq \emptyset\}$$
is a thickly-syndetic set.
Choose $a\in C$. Then for any $(i,j)\in \{0,1,\ldots,\ell\}\times \{0,1,\ldots,L\}$ one has
$$V_{i,j}:=V\cap (f(k_i+j)T^{a-i})^{-1}V$$
is a non-empty open subset of $V$ and
$$f(k_i+j)T^{a-i}V_{i,j}\subset V.$$

Write $p_{i,j}(n)=f(k_i+j)^{-1}f(k_i+j+n)f(n)^{-1}$ for any $(i,j)\in \{0,1,\ldots,\ell\}\times \{0,1,\ldots,L\}$.
Let $$A^L=\{ p_{i,j}: (i,j)\in \{0,1,\ldots,\ell\}\times \{0,1,\ldots,L\}\}.$$
Then $A^L\subset \PG_0^*$ is a system and by Corollary \ref{leibman}
$$\phi(A^L)\prec \phi(\{f\}) \preccurlyeq \phi(A).$$
Hence $A^L$ precedes $A$.
By the inductive assumption, $X$ is $A^L_\D$-syndetic transitive.
Thus
$$D:=\{n \in \mathbb{Z}: V\cap \bigcap_{(i,j)\in \{0,1,\ldots,\ell\}\times \{0,1,\ldots,L\}} (p_{i,j}(n))^{-1}V_{i,j}\neq \emptyset\}$$
is a syndetic set.

For $m\in D$, there exists $x_m\in V$ such that $p_{i,j}(m)x_m\in V_{i,j}$
for any $(i,j)\in \{0,1,\ldots,\ell\}\times \{0,1,\ldots,L\}$.
Let $y_m={f(m)^{-1}}x_m$. Since $X=\bigcup_{i=0}^{\ell}T^{i}U$, there are $z_m\in U$ and $0\le b_m\le \ell$ such that $T^{a}y_m=T^{b_m}z_m$.
Then $z_m=f(m)^{-1}T^{a-b_m}x_m$ and we have
\begin{equation*}
\begin{split}
{f(m+k_{b_m}+j)}z_m&=f(m+k_{b_m}+j)f(m)^{-1}T^{a-b_m}x_m\\
&  = f(k_{b_m}+j) T^{a-b_m}\big(f(k_{b_m}+j)^{-1}f(k_{b_m}+j+m)f(m)^{-1}x_m\big)\\
&=f(k_{b_m}+j) T^{a-b_m}(p_{k_{b_m},j}(m)x_m)\\
& \in f(k_{b_m}+j)T ^{a-b_m}V_{b_m,j}\subset  V
\end{split}
\end{equation*}
for each for $j\in \{0,1,\ldots,L\}$.
Thus $$\{ m+k_{b_m}+j:0\le j\le L\} \subset N_f(U,V).$$
Hence the set $\{ n\in \mathbb{Z}: n+j\in N_f(U,V) \text{ for any }j\in \{0,1,\ldots,L\}\}$
contains the syndetic set $\{m+k_{b_m}:m\in D\}$.
As $L\in \mathbb{N}$ is arbitrary, $N_f(U,V)$ is a thickly-syndetic subset of $\mathbb{Z}$.

\subsubsection{Claim: $(X,\Gamma)$ is $A_\D$-syndetic transitive}\
\medskip

Let $A=\{f_1,\ldots, f_v\}$. Then $f_1,f_2,\ldots, f_v$ are distinct $\Gamma$-polynomials.
It remains to prove that for any given non-empty open sets $U, V_1, \ldots, V_v$ of $X$
\begin{equation}\label{thm-4.2-1}
N_A(U,V_1,\ldots,V_v):=\{n\in \mathbb{Z}:U\cap {f_1(n)^{-1}}V_1\cap \ldots \cap {f_v(n)^{-1}}V_v\not=\emptyset\}
\end{equation}
is a syndetic set.
\medskip

Let $T\in \Gamma$ be an element from the center of $\Gamma$ with $T\neq e_\Gamma$. As $(X,T)$ is minimal,
there is some $\ell\in \N$ such that $X=\bigcup_{i=0}^\ell T^{i}U$. Let $f\in A, f\neq e_\Gamma$, be a $\Gamma$-polynomial of weight minimal in $A$:
$w(f)\le w(f_j)$ for any $j=1,\ldots,v$. Without loss of generality assume that $f=f_1$. By Lemma \ref{freedom}
and Lemma \ref{lem-equi2}, there are $\{k_j\}_{j=0}^\ell \subseteq \N$
and $V_t^{(\ell)}\subset V_t$ for $t=1,\ldots, v$ such that
\begin{enumerate}
\item  For $t=1,\ldots, v$, $f_t(k_j) T^{-j}V_t^{(\ell)}\subset V_t, \quad \forall \ 0\le j\le \ell.$

    \medskip

\item $f_t(k_j)^{-1}f_t(n+k_j)f(n)^{-1}$ and $f_s(k_i)^{-1}f_s(n+k_i)f(n)^{-1}$ are distinct $\Gamma$-polynomials with respect to $n$
for any $t\neq s \in \{1,2,\cdots,v\}$ and $i,j\in \{0,1,\ldots,\ell\}$.

\medskip

\item
$f_t(k_j)^{-1}f_t(n+k_j)f(n)^{-1}\in \PG_0^*$
for any $ t \in \{2,\cdots,v\}$ and $j\in \{0,1,\ldots,\ell\}$.
\end{enumerate}
\medskip

Since it may happen that for some $t$ there are $i\not= j$ with
$$ f_t(k_j)^{-1}f_t(n+k_j)f(n)^{-1}= f_t(k_i)^{-1}f_t(n+k_i)f(n)^{-1}$$ for all $n\in \mathbb{Z}$,
for each $t \in \{1, 2,\cdots,v\}$ let $I_t\subset \{0,1,\ldots,\ell\}$ such that the elements of
$\{f_t(k_j)^{-1}f_t(n+k_j)f(n)^{-1}: j\in I_t\}$ are distinct $\Gamma$-polynomial in $\PG_0^*$ and
$$\{f_t(k_j)^{-1}f_t(n+k_j)f(n)^{-1}: j\in I_t\}=\{f_t(k_j)^{-1}f_t(n+k_j)f(n)^{-1}: j=0, 1,\ldots, \ell\}\setminus \{e_\Gamma(n)\},$$
where $e_\Gamma(n)$ is the constant $\Gamma$-polynomial with value $e_\Gamma$. Note that $I_1=\emptyset$
if and only if $f_1(k_j)^{-1}f_1(n+k_j)f(n)^{-1}\equiv e_\Gamma$ for each $j\in \{0,1,\ldots,\ell\}$.
Moreover, by the above condition (3),
$|I_t|\ge 1$  and
$$\{f_t(k_j)^{-1}f_t(n+k_j)f(n)^{-1}: j\in I_t\}=\{f_t(k_j)^{-1}f_t(n+k_j)f(n)^{-1}: j=0, 1,\ldots, \ell\}$$
for any $t\ge 2$.













Let
\begin{equation*}
\begin{split}
  A'=\bigcup_{t=1}^v  \{f_t(k_j)^{-1}f_t(n+k_j)f(n)^{-1}: j\in I_t\}.
\end{split}
\end{equation*}
 Then $ A'\subseteq \PG_0^*$ and by Corollary \ref{leibman}, $A'$ precedes $A$.
According to the inductive assumption, $X$ is $A'_\D$-syndetic transitive. Hence
$$E:=\{ m\in \mathbb{Z}: V_1^{(\ell)}\cap \bigcap_{t=1}^v \bigcap_{j\in I_t} (f_t(k_j)^{-1}f_t(m+k_j)f(m)^{-1})^{-1}V_t^{(\ell)}\neq \emptyset\}$$
is a syndetic set.

For $m\in E$, there are $x_m\in V_1^{(\ell)}$  such that
$$f_t(k_i)^{-1}f_t(m+k_i)f(m)^{-1} x_m \in V_t^{(\ell)}, \text{ for any }\ 1\le t\le v,\ i\in I_t.$$
Moreover by the choice of $I_t$ and $x_m\in V_1^{(\ell)}$, one has that
$$f_t(k_j)^{-1}f_t(m+k_j)f(m)^{-1} x_m\in V_t^{(\ell)}, \text{ for any }\ 1\le t\le v,\ j\in \{0,1,\ldots,\ell\}.$$

\medskip
Let $y_m={f(m)^{-1}}x_m$. Since $X=\bigcup_{i=0}^{\ell}T^{i}U$, there is $z_m\in U$ and $0\le b_m\le \ell$ such that $y_m=T^{b_m}z_m$.
Then $z_m=T^{-b_m}f(m)^{-1}x_m$ and we have
\begin{equation*}
\begin{split}
{f_t(m+k_{b_m})}z_m&=f_t(m+k_{b_m})T^{-b_m}f(m)^{-1}x_m\\ &  = f_t(k_{b_m})T^{-b_m}\big(f_t(k_{b_m})^{-1}f_t(m+k_{b_m})f(m)^{-1}\big)x_m\\
& \in f_t(k_{b_m})T ^{-b_m}V_t^{(\ell)}\subset  V_t
\end{split}
\end{equation*}
for each $1\le t\le v$.
This implies that $$z_m\in  U\cap {f_1(n)^{-1}}V_1\cap \ldots \cap {f_v(n)^{-1}}V_v$$ with $n=m+k_{b_m}$.
Thus
$$N_A(U,V_1,\ldots,V_v)\supseteq \{m+k_{b_m}:m\in E\}$$
is a syndetic set. Hence the proof of the whole theorem is completed.

\subsection{Proof of Theorem \ref{ye00}}
\
\medskip

We have the following two cases:

\medskip
\noindent {\bf Case 1:}  $g(n)=(g(1))^n$ for any $n\in \mathbb{Z}$. Then $g(1)\neq e_\Gamma$ as $g \not \equiv e_\Gamma$.
Since $(X, g(1))$ is minimal,
for each $x\in X$ and each non-empty open subset $U$ of $X$, $\{n\in\Z: g(n)x\in U\}$ is syndetic.

\medskip
\noindent {\bf Case 2:} There exists $v\in \mathbb{Z}$ such that $g(v)\neq (g(1))^v$. Thus
$g(u+1)\neq g(u)g(1)$ for some $u\in \mathbb{Z}$.
Let $f(n)=g(n)^{-1}g(n+1)g(1)^{-1}$ for $n\in \mathbb{Z}$. Then $f(u)\neq e_\Gamma$ and so
$f\in \PG_0^*$.

Assume that the weight of the $\Gamma$-polynomial $f(n)=\prod_{j=1}^sS_j^{p_j(n)}$ is $(l,k)$.
Then $k>0$ and the degree of the integral polynomial $p_\ell$ is $k$  as $f\in \PG_0^*$.
Thus there exists $M\in \mathbb{N}$ such that $p_\ell$ is strictly monotone on $[M,+\infty)$.
Particularly, for any $i,j\ge M$ with $i\neq j$ we have $p_\ell(i)\neq p_\ell(j)$ and hence $f(i)\neq f(j)$ by Theorem \ref{thm-5.1}(2).

\medskip
By Theorem \ref{stronger-thm} for each given non-empty open subset $V$ of $X$,
$$F=:\{n\in\Z: V\cap g(n)^{-1}V\not=\emptyset\}=\{\ldots <n_{-1}<n_0<n_1<\ldots\}$$
is thickly syndetic, where we require $n_0\ge M$. Since $F$ is syndetic, there is $L(V)\in \mathbb{N}$
such that $$n_{i+1}-n_i\le L(V)$$ for $i\in \mathbb{N}$.
Consider $\Gamma$-polynomials
$\{g(n_i)^{-1}g(n+n_i): i\in\Z\}$. Since $g\in \PG_0^*$, $$g^{-1}(n_i)g(n+n_i)\in \PG_0^*$$ for any $i\in \Z$.

Now for any $i,j\in \mathbb{N}$ with $i\neq j$, note that
\begin{align*}
g(n_i)^{-1}g(1+n_i)=f(n_i)g(1)\neq f(n_j)g(1)=g(n_j)^{-1}g(1+n_j),
\end{align*}
hence $g(n_i)^{-1}g(n+n_i)$ and $g(n_j)^{-1}g(n+n_j)$ are distinct $\Gamma$-polynomials in $\PG_0^*$.

For $d\in\mathbb{N}$, let $A_d(V)$ be the set of all points $y_d\in X$ such that we can find $m_1<\cdots<m_d\in\Z$
satisfying $g(m_j)y_d\in V$ for $j=1,\cdots,d$ and $m_{j+1}<m_{j}+L(V)$ for $j=1,\ldots,d-1$.

For any $i\in\Z$, let $V_i=V\cap g(n_i)^{-1}V$. Then $V_i$ is a non-empty open subset of $V$ and
$g(n_i)V_i\subset V$.
Let $U$ be a non-empty open subset of $X$.
Applying Theorem \ref{main-thm}, there are $y_d\in U$ and $l_d\in\Z$ such that
$g^{-1}(n_{j})g(l_d+n_{j})y_d\in V_{j}$ which implies that $g(l_d+n_{j})y_d\in V$ for $j=1,\ldots,d$.
Thus $y_d\in U\cap A_d(V)$.

By what we just proved, $A_d(V)$ is an open dense subset of $X$ as $U$ is arbitrary. Assume that $\{U_i\}$ is a base of the
topology of $X$ and $X_0=\bigcap_{i\in \N}\bigcap_{d\in\N} A_d(U_i)$. We claim $X_0$ is the set we need.
In fact, for any non-empty open subset $U$ of $X$, there is $i\in \N$ with $U_i\subset U$.
So for $x\in X_0$, $x\in  \bigcap_{d\in\N} A_d(U_i)$. Thus, for any $d\in \mathbb{N}$, there are $m_1^d<m_2^d<\cdots<m_d^d\in\Z$ such that
$m_{j+1}^d-m_j^d\le L(U_i)$ for $j=1,\cdots,d-1$ and
$$N_g(x,U)\supset N_g(x,U_i)\supset \bigcup_{d\in\N}\{ m_1^d,m_2^d,\cdots,m_d^d\}$$
i.e. $\{n\in\Z:g(n)x\in U\}$ is piecewise syndetic.





\end{document}